\documentclass[12pt,reqno]{amsart} 
\usepackage[margin=1.5in]{geometry}
\usepackage[utf8]{inputenc}
\usepackage{amsmath}
\usepackage{amssymb}
\usepackage{amsfonts}

\usepackage{dsfont}
\usepackage{amsthm}
\usepackage{mathrsfs} 
\usepackage{xcolor}
\usepackage{hyperref}
\usepackage{xparse}

\usepackage{centernot}
\usepackage{enumerate}
\usepackage[shortlabels]{enumitem}
\newcommand{\R}{\mathbb{R}}
\newcommand{\Z}{\mathbb{Z}}
\newcommand{\Q}{\mathbb{Q}}

\newcommand{\C}{\mathbb{C}}

\newcommand{\T}{\mathbb{T}}

\newcommand{\Prob}{\mathbb{P}}

\newcommand{\ep}{\epsilon}

\DeclareMathOperator{\supp}{Supp}
\DeclareMathOperator{\Blue}{Blue}
\DeclareMathOperator{\Red}{Red}

\hypersetup{
    colorlinks=true,
    linkcolor=blue,
    filecolor=magenta,
    urlcolor=blue,
    citecolor=blue
}

\newtheorem{thm}{Theorem}

\newtheorem{result}{Result}[section]
\newtheorem{lem}[result]{Lemma}
\newtheorem{prp}[result]{Proposition}

\theoremstyle{definition}
\newtheorem*{clm}{Claim}

\theoremstyle{remark}
\newtheorem{rmk}[result]{Remark}
\newtheorem*{ack}{Acknowledgements}

\numberwithin{equation}{section}
\newcommand\numberthis{\addtocounter{equation}{1}\tag{\theequation}}
\title{Improved lower bounds for van der Waerden numbers}
\author{Zach Hunter}
\date{\today}

\begin{document}

\maketitle
\begin{abstract}
    Recently, Ben Green proved that the two-color van der Waerden number $w(3,k)$ is bounded from below by $k^{b_0(k)}$ where $b_0(k) = c_0\left(\frac{\log k }{\log \log k}\right)^{1/3}$. We prove a new lower bound of $k^{b(k)}$ with $b(k) =  \frac{c\log k}{\log \log k}$. This is done by modifying Green's argument, replacing a complicated result about random quadratic forms with an elementary probabilistic result.
\end{abstract}

\section{Introduction}

In this paper, we will be concerned with bounding the two-color van der Waerden numbers $w(3,k)$, defined as follows. For any $k\ge 3$, we let $w(3,k)$ denote the smallest $N$ such that for any blue-red coloring of $[N]:= \{1,\dots,N\}$, there either exists a blue arithmetic progression of length $3$, or a red arithmetic progression of length $k$. In terms of upper bounds, the best known result is of the form $w(3,k) < e^{k^{1-c}}$ for some $c> 0$, which was first shown by Schoen \cite{schoen}. This now also follows from the bound on Roth's theorem proven by Bloom and Sisask \cite{bloom}, by a simple density argument\footnote{Observe that for any blue-red coloring of $[N]$ lacking red progressions of length $k$, we must have that the set of blue integers has at least $ N/k-1$ elements, as at least one of every $k$ consecutive integers must be colored blue. So when $N = e^{k^{1-c}}$ for some sufficiently small $c> 0$, we will have (due to the bound on Roth's theorem in \cite{bloom}) that our set of blue integers must be too dense in $[N]$ to lack progressions of length 3.}. 

In \cite[Theorem~1.1]{green}, Green showed that 
$w(3,k) \ge k^{b_0(k)}$, where $b_0(k) = c_0 \left(\frac{\log k}{\log \log k}\right)^{1/3}$ for some $c_0 > 0$. An equivalent way to view the result is as follows. Let $f(N)$ be the smallest $k$ such that $w(3,k) > N$, so that there exists a blue-red coloring coloring of $[N]$ with no blue arithmetic progression of length 3 and no red arithmetic progression of length $f(N)$. The result of \cite{green} implies that $f(N)\le e^{C_0(\log N)^{3/4}(\log\log N)^{1/4}}$ for some $C_0 > 0$.

In this paper, we improve the lower bound of $w(3,k)$ as follows.
\begin{thm}\label{main} For some absolute constants $C,c>0$ the following holds. We have $w(3,k) \ge k^{b(k)}$, where $b(k) = \frac{c\log k}{\log \log k}$. Equivalently, $f(N) \le e^{C(\log N)^{1/2}(\log\log N)^{1/2}}$.
\end{thm}
\begin{rmk}\label{supposed upper bound}In the end of \cite[Section~2]{green}, Green stated that it is reasonable to believe $w(3,k) \le k^{O(\log k)}$ or equivalently $f(N) \ge  e^{c (\log N)^{1/2}}$. This is what we would achieve if we could color $[N]$ so the the blue set has size $Ne^{-\Theta(\sqrt{\log N})}$ (like in Behrend's construction) and the red set ``behaved randomly''.

\end{rmk}
Since Theorem~\ref{main} proves $w(3,k) \ge k^{(\log k)^{1-o(1)}}$, Remark~\ref{supposed upper bound} suggests our result is likely to be essentially best possible.

This improvement in Theorem~\ref{main} follows from a few adjustments to the argument of \cite{green}. Said adjustments also happen to greatly simplify the proof of our main theorem. In particular, we replace the rather involved proof of \cite[Proposition~5.4]{green} which was done in \cite[Sections~9-16]{green}, with a short probabilistic proof.

An additional goal of this paper is to be a self-contained and (relatively) accessible resource about lower-bounding $w(3,k)$. Given how our methods build upon the work of \cite{green}, we reproduce several results from \cite{green} for the sake of completeness. This is done with the permission and encouragement of Green, and in all such instances we properly reference which parts of \cite{green} were reproduced. 

\begin{ack}The author would like to thank Ben Green for many helpful comments concerning the presentation of this paper. The author also thanks Zachary Chase for his encouragement, and for checking a draft of this paper. Lastly, the author is grateful for the helpful reports of two anonymous referees.
\end{ack}

\subsection{Sketch}

The essential ideas are as follows. Since the work of Behrend \cite{behrend}, the largest known constructions of large 3-AP-free\footnote{For positive integer $k$, we will refer to arithmetic progressions of length $k$ as $k$-APs.} sets have been based on the observation that $d$-dimensional spheres do not contain 3 collinear points (see e.g., \cite{elkin,wolf}). These constructions roughly work by first establishing a map $\varphi$ from $S\subset [N]$ to $d$-dimensional space, such that 3-APs in $S$ are mapped to certain collinear triples; one then uses spheres or annuli to color $\varphi(S)$ avoiding such collinear triples.

We will be working with the framework from \cite{wolf}, which considered thin annuli in a torus. Here, one picks a dimension $D$ and considers the $D$-dimensional torus $\T^D = \R^D/\Z^D$. Given $\theta \in \T^D$, and a blue-red coloring $f:\T^D\to \{\Blue,\Red\}$, we can obtain a blue-red coloring $F=F_{f,\theta}$ of $[N]$ by associating $n \in [N]$ with $\theta n \in \T^D$ and assigning $F(n) = f(\theta n)$.

We shall now describe the construction from \cite{green} and our modification. Afterwards we will discuss the ``meaning'' behind these constructions and why our new construction was able to yield better bounds.
\subsubsection{Constructions}

In \cite{green}, one first fixes a well-distributed set of points $x_1,\dots,x_M \in \T^D$ (such a set will occur with high probability when choosing these points uniformly at random). We then randomly create a blue-red coloring $f$ so that $f^{-1}(\Blue)$ is the set of translates\[\bigcup_{i \in [M]} (x_i + \mathcal{A}) \]where $\mathcal{A}$ is the projection to $\T^D$ of a random ellipsoidal annulus with small radius and thin width. Lastly, one chooses $\theta \in \T^D$ uniformly at random and colors $[N]$ with $F_{f,\theta}$. With high probability, this process ``succeeds'', yielding a coloring of $[N]$ that lacks blue 3-APs and long red APs.

In our new construction, we again fix a well-distributed set of points $x_1,\dots,x_M \in \T^D$ and choose $\theta \in \T^D$ uniformly at random. However, we shall randomly create $f$ so that \[f^{-1}(\Blue) = \bigcup_{i \in [M]} (x_i +\mathcal{A}_i)\]where each $\mathcal{A}_i$ is the projection of a thin annulus with a random small radius (these radii being chosen independently). As before, $[N]$ is colored by $F_{f,\theta}$, and with high probability we ``succeed''.

To summarize the differences, Green used translates of a single randomly determined (ellipsoidal) annulus (with random eccentricity and fixed radius), while our construction uses translates of (circular) annuli each being determined independently at random (with fixed eccentricity and random radius). But more importantly (as will be elaborated in Subsection~\ref{how}), Green only considers his single instance of randomness, while we take advantage of having multiple instances of randomness (additionally, we make more careful use of how $N,D,M$ grow in relation to one another, see Subsection~\ref{random annuli}). Indeed, one can modify this write-up to show (via a slightly less natural argument) that the same bound as Theorem~\ref{main} can be achieved using translates of a fixed annuli by exploiting the randomness of $x_i$.

\subsubsection{Why (roughly) these colorings work}\label{how}

We emphasize that the following two subsections are sketches, and details are omitted for ease of exposition.

It will be relatively straightforward to show that these constructions lack blue 3-APs with high probability. Indeed, an observation from \cite{wolf} is that a single translate of a thin\footnote{Here and before, ``thin'' is defined with respect to $N$.} annulus with radius $\le 1/4$ will be 3-AP-free with high probability. With the centers $x_1,\dots, x_M$ being well-distributed, it will continue to be the case that the coloring (which uses multiple translates of thin annuli) is free of blue 3-APs with high probability.

At the same time, one can prove that with high probability, these constructions will lack long red progressions. In \cite{green}, this was done in two steps. First, it is shown that because the centers of our annuli are well-distributed, any long (possibly red) AP will have a point which is very close to one of these centers. In \cite[Proposition~5.4]{green}, it was then shown that there exists an ellipsoidal annulus so that any long progression which has a point very close to the center must intersect this annulus.

Meanwhile, our construction prevents long red arithmetic progressions as follows. We observe that a slightly stronger conclusion can be deduced from the first step in \cite{green}. Namely, because the centers of our annuli are well-distributed, for any long (possibly red) AP we can find \textit{many} centers which are very close to points in the AP. The radii of our annuli are chosen randomly in such a way that whenever a point is very close to a center, there is a non-zero chance chance it will be colored blue. Because of how many centers are very close to points in the AP, the chance none of the points in the AP are blue is exceedingly small, preventing long red APs with high probability.

Now, before we get into the next subsection, we wish to comment on why using multiple translates of annuli is useful here (since using a single annulus/sphere typically sufficed for Behrend-esque constructions in prior literature). At the most basic level, this helps make the red complement look more random, which happens to be useful. Concretely, we are concerned about progressions that are concentrated around some subtorii with small codimension.

\subsubsection{Analysis of quantifiers}

Given $x_0,\alpha \in \T^D$ and a positive integer $X$, we write $(x_0,\alpha,X)$-AP to denote the set $\{x_0,x_0+\alpha,\dots,x_0+(X-1)\alpha\}\subset \T^D$. Note that the $X$-AP $P = \{n_0,n_0+d,\dots,n_0 + (X-1)d\}\subset [N]$ will be monochromatic under $F_{f,\theta}$ if and only if the $(\theta n_0,\theta d,X)$-AP is monochromatic under $f$.

Thus, to show $F$ is free of monochromatic arithmetic progressions of length $X$, it would suffice to show that $f$ does not monochromatically color the $(x_0,\alpha,X)$-AP for any $x_0,\alpha \in  \{\theta,2\theta,\dots,N\theta \}$. However, it will be more natural to prove stronger statements where we do not impose restrictions on the values $x_0$ will take\footnote{A reason for this is that the orbit $\{\theta,2\theta, \dots,N\theta\}$ should be very well-distributed for typical $\theta \in \T^D$. Thus for any $x_0 \in \T^D$ we should be able to find $x$ in the orbit which closely approximates $x_0$, and these should yield very similar progressions. Meanwhile, slightly altering the value of $\alpha$ can significantly change long progressions (a good example is when $\alpha = 0$), thus it is still important to control the values $\alpha$ can take.}.

For example, both aforementioned constructions will (with high probability) prevent 3-APs in a ``uniform'' sense. Given $\theta \in \T^D$, we have a set of possible common differences $S_\theta :=  \{\theta,2\theta,\dots,N\theta \}$. It will be the case that if $\alpha$ does not have an exceptionally small norm, then for all choices of $x_0 \in \T^D$ the $(x_0,\alpha,3)$-AP will not be monochromatically blue under $f$ . Choosing $\theta \in \T^D$ uniformly at random we have that $d\theta$ is uniformly distributed for each $d \in [N]$, and then by a union bound we get that with high probability there will be no $\alpha \in S_\theta$ with an exceptionally small norm. 
 
In \cite{green}, red progressions are also prevented in a uniform sense. Meaning that choosing $\theta \in \T^D$ randomly and then choosing a random ellipsoidal annulus $A$ to create $f$, with high probability for all $\alpha \in S_\theta$, there will be no $x_0 \in \T^D$ such that the $(x_0,\alpha,X)$-AP is monochromatically red under $f$ (where $X= N^{O(1/\sqrt{D})})$.

While avoiding blue 3-APs in a uniform sense is quite manageable, doing the same for long red APs was very tricky and caused a bottleneck for the lower bound in \cite{green}. Thus, it seemed desirable to circumvent avoiding red APs in such a strong sense. The construction from this paper manages to do this.

In particular, we will show for typical $\theta$, that for any fixed choice of $x_0 \in \T^D, \alpha \in S_\theta$ when we randomly determine $f$ by choosing the radii of our annuli, the probability that the $(x_0,\alpha,X)$-AP  (where $X = N^{O(1/D)}$) is monochromatically red under $f$ is very small (less than $N^{-3}$, so that with high probability, none of the choices of $x_0,\alpha$ corresponding to the $<N^2$ progressions in $[N]$ occur). 

Another way to think about the above is in terms of Geometric Ramsey Theory (see e.g. \cite{conlon}). Given a subset $\Omega$ of $\T^D$ and $\alpha \in \T^D$, we say that $\Omega$ is $(\cdot, \alpha,X)$-AP-free if for every $x \in \T^D$ the $(x,\alpha,X)$-AP is not contained in $\Omega$. In the construction of \cite{green}, when $f$ is chosen randomly, $f^{-1}(\Blue)$ is guaranteed to be $(\cdot,\alpha,3)$-free for a set of $\alpha$ with measure $\ge 1-1/N^2$, and with high probability $f^{-1}(\Red)$ will be $(\cdot,\alpha,X)$-free on a set of $\alpha$ with measure $\ge 1-1/N^2$. And thus choosing $\theta$ randomly, with high probability there will not be $\alpha\in S_\theta$ where $f^{-1}(\Blue)$ (respectively $f^{-1}(\Red)$) is not $(\cdot,\alpha,3)$-free (respectively $(\cdot,\alpha,X)$-free). The construction of this paper does something which more resembles a half-multiplicity analogue (half-multiplicity Ramsey numbers for graphs have recently been considered in \cite{sawin}), where $f^{-1}(\Blue)$ is guaranteed to be $(\cdot,\alpha,3)$-free for some set of $\alpha$ with measure $\ge 1-1/N^2$, and for a set of $\alpha$ with measure $\ge 1-1/N^2$ the collection of $(x_0,\alpha,X)$-APs contained in $f^{-1}(\Red)$ is expected to be ``sparse'' in a sense.

\subsection{Outline}

Our new ideas are in Sections~\ref{setup} and \ref{final}. In Sections~\ref{bluecase} and \ref{diophantine} we reprove results from \cite{green} that are necessary but whose statements did not substantially change. In Appendix~\ref{ax lattice} we prove a slight modification of Lemma~A.2 from \cite[p. 62]{green}. For completeness, in Appendix~\ref{ax fourier} we include proofs of some lemmas whose statements are identical to results in \cite[Appendix~\ref{ax fourier}, p. 63-67]{green}.

More specifically, in Section~\ref{setup}, we establish the magnitudes of various parameters, and then formally go over how we will construct colorings of $[N]$ with random translates of annuli. Then in Section~\ref{bluecase} we handle blue progressions. Next in Section~\ref{diophantine} we confirm that ``diophantine'' $\theta \in \T^D$ are still typical when using our modified choice of parameters. In Section~\ref{final}, we handle the red progressions.

\subsection{Definitions}

We will reuse the notation from \cite{green}.

For positive integer $n$, we shall write $[n]:=\{1,\dots,n\}$.

For positive integer $D$, we will write $\T^D$ to denote the $D$-dimensional torus which is $\R^D/\Z^D$. Naturally, we shall use $\T$ as shorthand for $\T^1$.

For $x\in \T$, we will write $||x||_\T= \min_{n\in\Z}|x-n|$, the distance from $x$ to the nearest integer. For $x=(x_1,\dots,x_D)\in \T^D$, we will write $||x||_{\T^D} =\max_{i\in [D]}||x_i||_\T$.

For $\xi = (\xi_1,\dots,\xi_D)\in \Z^D$, we will write $|\xi|:=\max_{i\in [D]}|\xi_i|=||\xi||_\infty$.

Outside of Appendix~\ref{ax fourier}, $\pi$ will always be used to denote the natural projection $\pi:\R^D\to\T^D$ given by $x\mapsto x+\Z^D$. Despite $\pi$ not being injective, as an abuse of notation we will write $\pi^{-1}(x)$ to denote the unique $y
\in (-1/2,1/2]^D$ satisfying $\pi(y) = x$.

Finally, given $\ep >0$, we will write $B_\ep(0)$ to denote the Euclidean ball in $\R^D$ of radius $\ep$ around $0$, that is to say $B_\ep(0):= \{x\in \R^D: ||x||_2 <\ep \}$.

\subsubsection{Fourier Transforms}

We will take Fourier transforms of functions on $\Z^D$ and $\T^D$. We will always denote this with the hat symbol, which we define as follows:
\begin{itemize}
    \item If $f:\Z^D\to \C$, then $\hat{f}(\theta) = \sum_{n \in \Z^D} f(n) e(-n\cdot \theta) $ for $\theta \in \T^D$.
    \item If $f:\T^D\to \C$, then $\hat{f}(\xi) = \int_{\T^D} f(x) e(-\xi\cdot x) $ for $\xi \in \Z^D$.
\end{itemize}
We will only be considering the Fourier transform of smooth, rapidly decaying functions where convergence is not an issue.

\section{Setup}\label{setup}
We shall take $c = 0.01$. The exact value of $c$ is not very important, one need only remember that $c> 0$ but is not too large.

We will also use absolute constants $C_1,C_2>0$ which are determined later on. It will suffice to take $C_1 \ge 12$ and $C_2 \ge 8c^{-1}C_1=800C_1$. 

We will have a parameter $D$, which we will be assuming is sufficiently large. Let\footnote{We will often omit writing the floor function for large constants that are supposed to be integers.} $\rho = 1/D^4,M = \rho^{-(1/4+c)D}$ and $N = D^{cD^2/2}$.

We will sometimes write $X$ to denote $N^{100(C_2+2)/D}$, and $Y$ to denote $\rho^{-cD}$. We also shall write $K$ to denote $ \rho N^{4/D} $.

For reference, the following relationships between the size of parameters are intended, and should be kept in mind. We have that $K,Y,M,\rho^{-D},X$ are all of shape $D^{\Theta(D)}$ with 
\[\log K < \log Y <\log M < \log \rho^{-D} <\log X,\] and we have designed our parameters so that each of the ratios $\frac{\log Y}{\log K},\frac{\log M}{\log Y},\frac{\log \rho^{-D}}{\log M},\frac{\log X}{\log \rho^{-D}}$ is sufficiently large (for example, in the proof of Proposition~\ref{red} we wanted to have $Y\ge K^2$ so that $(1-1/K)^Y$ was very small, thus we chose our constants to ensure this). Meanwhile, we also have $X = N^{O(1/D)}$, and made sure $\frac{\log X^{D/100}}{\log N}$ is sufficiently large.

\subsection{Well-distributed centers}

In the rest of the paper, we will fix $\rho = D^{-4}$. However, in the following proposition we need only assume $\rho< 1/D^2$.

\begin{prp}
Suppose $D$ is sufficiently large, and $\rho<1/D^2$. There exist $x_1,\dots,x_M \in \T^D$ such that the following holds.
\begin{enumerate}\label{welld}
    \item For $i_1,i_2,i_3\in [M]$, either $||x_{i_1}-2x_{i_2}+x_{i_3}||_{\T^D}>10\rho $ or $i_1 = i_2 = i_3$.
\vspace{1em}
    \item Let $V\le \Q^D$ be any subspace with dimension at most $D/100$. For every $x^* \in \T^D$, there exist distinct $j_1,\dots,j_Y\in [M]$ and (not necessarily distinct) $u_1,\dots,u_Y\in \pi(B_{\rho/10}(0))$ such that $||\xi\cdot(x_{j_k}+u_k-x^*)||_\T <1/100$ for all $\xi \in V\cap \Z^D$ with $|\xi|\le \rho^{-3}$ and all $k \in [Y]$.
\end{enumerate}
\begin{rmk}
We mostly follow the proof of Proposition~4.1 from \cite[p. 11-14]{green}, but with two adjustments to Condition 2. First, we introduce an error $u \in \pi(B_{\rho/10}(0))$ and prove a new claim to utilize this error. This adjustment allows us to show Condition 2 holds for all $V$ with dimension $\le D/100$, improving upon \cite{green} where it could only be shown to hold for $V$ with dimension $\le \ep \sqrt{D}$ for some sufficiently small $\ep > 0$. Our second adjustment is observing that ``many'' choices of centers work. This strengthened conclusion requires no new effort to prove, but will be useful later on.
\end{rmk}
\begin{proof}

Recall $Y = \rho^{-cD},M = \rho^{-(1/4+c)D} = \rho^{-D/4}Y$. We will select $x_1,\dots,x_M\in \T^D$ independently and uniformly at random. Let $F_1$ (respectively $F_2$) denote the event that the selected $x_1,\dots,x_M$ fail to satisfy Condition 1 (respectively Condition 2). We will show that (for sufficiently large $D$) $\Prob(F_1),\Prob(F_2) <1/4$, hence with probability at least 1/2, $x_1,\dots,x_M$ satisfy the necessary conditions.

We first bound $\Prob(F_1)$. For any $(i_1,i_2,i_3)\in [M]^3$  where $i_1,i_2,i_3$ are not all equal, we have that $x_{i_1}-2x_{i_2}+x_{i_3}$ is uniformly distributed on $\T^D$. Thus, $\Prob(||x_{i_1}-2x_{i_2}+x_{i_3}||_{\T^D} \le  10\rho )=(20\rho)^D$. By a union bound over all $<M^3$ such choices of $(i_1,i_2,i_3)$, the first condition fails with probability at most $M^3(20\rho)^D <1/4$ for large $D$. 

We now turn to bounding $\Prob(F_2)$. Here we will be repeatedly applying union bounds. Due to the size of $M$, we may partition $[M]$ into $Y$ parts of size $\rho^{-D/4}$, $S_1,\dots,S_Y$. We fix such a partition. Given a selection $x_1,\dots,x_M$ and $y\in [Y]$, we say the selection is \textit{$y$-bad} if there exists a choice of $V\le \Q^D,\, \dim V\le D/100$ and $x^*\in \T^D$ so that for each $j \in S_y,u \in \pi(B_{\rho/10}(0))$, there exists $\xi\in V\cap \Z^D$ with $|\xi|\le \rho^{-3}$ such that $||\xi\cdot (x_j+u-x^*)||_\T\ge 1/100$ (in which case we say $V,x^*$ demonstrate $y$-badness).

Clearly, for Condition 2 to fail, there must be some $y\in [Y]$ where our selection is $y$-bad (otherwise, since the sets $S_1,\dots,S_Y$ are disjoint, we may choose $j_y \in S_y,u_y\in \pi(B_{\rho/10}(0))$ for $y\in Y$). For $y \in [Y]$, let $E_y$ be the event our selection is $y$-bad. Since $x_1,\dots,x_M$ are selected independently, and each set $S_y$ satisfies $|S_y| = \rho^{-D/4}$, we have that each $E_y$ has the same distribution as $E_1$. Thus, by a union bound, $\Prob(F_2) \le Y\Prob(E_1)$. It remains to bound $\Prob(E_1)$.

We shall assume $V$ is spanned (over $\Q$) by vectors $\xi \in \Z^D$ with $|\xi|\le \rho^{-3}$ (since otherwise, we may pass from $V$ to the subspace spanned by these vectors). There are at most $\frac{D}{100}(3\rho^{-3})^{D^2/100}\le \rho^{-D^2}$ such spaces $V$ with dimension $\le D/100$ (choose the dimension $m \le D/100$, and then $m$ basis vectors $\xi \in\Z^D$ with $|\xi|\le \rho^{-3}$).

We fix such a $V$, and let $E_1(V)$ denote the event that there exists some $x^*\in \T^D$ such that $V,x^*$ demonstrate $E_1$ occurs. By Lemma~\ref{lattice gen}, there exists some choice of $\xi_1,\dots,\xi_m \in \Z^D, m = \dim V \le D/100$, with $|\xi_i|\le \rho^{-3}D \le \rho^{-4}$ for each $i\in [m]$, such that each $\xi \in \Z^D\cap V$ with $|\xi| \le \rho^{-3}$ can be written as a $\Z$-linear combination $\xi = \sum_{i=1}^m n_i \xi_i$ where the coefficients $n_i$ satisfy $|n_i|\le m!(D\rho^{-3})^m < \frac{1}{100m}\rho^{-7m}$. Consequently, we see $\xi_1,\dots,\xi_m$ are linearly independent and span $V$ (over $\Q$).

So given $x^*\in \T^D$, if there exists $j \in S_1,u \in \pi(B_{\rho/10}(0))$ such that $||\xi_i\cdot (x_j+u -x^*)||_\T \le \rho^{7m}$ for all $i\in \{1,\dots,m\}$, we will get that for each $\xi \in V\cap \Z^D$ with $|\xi|\le \rho^{-3}$\[||\xi(x_j+u-x^*)||_\T \le \rho^{7m}\sum_{i=1}^m n_i< \frac{1}{100}.\]Thus $E_1(V)$ can only occur if there is some $x^*\in \T^D$ where such $j \in S_1,u \in \pi(B_{\rho/10}(0))$ cannot be found.

For convenience, we define $\phi:\T^D\to \T^m$ so that $x\mapsto (\xi_1\cdot x,\dots, \xi_m\cdot x)$, and we remark that $\phi$ is a homomorphism.

We divide $\T^m$ into $\rho^{-7m^2}\le \rho^{-D^2}$ boxes of sidelength $\rho^{7m}$; for $E_1$ to occur there must be some box $B$ that does not contain $\phi(x_j+u)$ for any $j \in S_1,u \in \pi(B_{\rho/10}(0))$. We fix a box $B$.

Choosing $x\in \T^D$ uniformly at random, we note that $\phi(x)$ is uniformly distributed over $\T^m$. This is a well-known fact, but for completeness we mention a brief proof. Letting $f(t) = e(\gamma \cdot t)$ be any non-trivial character over $\T^m$, we will have $\gamma_1\xi_1+\dots+\gamma_m\xi_m\neq 0$ since $\{\xi_1,\dots,\xi_m\}$ are linearly independent over $\Z$, thus\[\int_{\T^D} f(\xi_1\cdot x,\dots,\xi_m\cdot x)\,dx = \int_{\T^D}e((\gamma_1\xi_1+\dots+\gamma_m\xi_m)\cdot x)\,dx= 0\] and so by Weyl's equidistibution theorem $(\xi_1\cdot x,\dots,\xi_m\cdot x)=\phi(x)$ is uniformly distributed.

Next, we require the following claim.
\begin{clm}
There is a finite set $U \subset \pi(B_{\rho/10}(0))$ with $|U| \ge \rho^{10m-7m^2}$ such that for any distinct $u,u'\in U$, $||\phi(u-u')||_{\T^m} \ge \rho^{7m}$. \begin{proof}
Since $||w||_2\le D^{1/2}||w||_\infty$ for $w\in \R^D$, and $\rho^5 \le \frac{\rho}{10D^{1/2}}$ for $D\ge 2$ (because $\rho<1/D^2$), we have that \[  W:= \{\rho^{7m}\eta: \eta \in \Z^D, |\eta|\le \rho^{5-7m}\}\subset B_{\rho/10}(0).\] For $w,w'\in W$ with $w= \rho^{7m}\eta,w' = \rho^{7m}\eta'$, we have that \[|\xi_i\cdot(w-w')| \le 2D\rho \le 2/D\le 1/2\] for all $i\in [m]$ (recalling $|\xi_i|\le \rho^{-4}$, $\rho<1/D^2$, and assuming $D\ge 4$), whence \[||\phi\circ \pi (w-w')||_{\T^m} = \rho^{7m}|(\xi_1\cdot (\eta-\eta'),\dots,\xi_m\cdot (\eta-\eta'))|.\]For distinct $v,v'\in V$, we have that $v-v' \in V\setminus \{0^D\}$ and thus $\xi_i \cdot (v-v') \neq 0$ for some $i\in [m]$ (since $V$ is spanned by $\xi_1,\dots,\xi_m$). Thus, for distinct $\eta,\eta' \in V\cap \Z^D$ with $|\eta|,|\eta'| \le \rho^{5-7m}$, letting $w = \rho^{7m}\eta,w'= \rho^{7m}\eta'$, we have $w,w'\in W$ and $||\phi\circ \pi (w-w')||_{\T^m} \ge \rho^{7m}$. 

Noting $W \subset (-1/2,1/2]^D$, we have that $|\pi(W')| = |W'|$ for any $W' \subset W$. From the above, for any $T \subset V\cap \Z^D$ with $|\eta| \le \rho^{5-7m}$ for all $\eta \in T$, we get that $W' :=\{\rho^{7m} \eta: \eta \in T \}$ is a subset of $W$ where $||\phi\circ\pi (w-w')||_{\T^m}\ge \rho^{7m}$ for distinct $w,w' \in W'$ (and it is obvious that $|W'| = |T|$). Thus it remains to find such a $T$ with $|T| \ge \rho^{10m-7m^2}$.

By the triangle inequality, $\left|\sum_{i=1}^m n_i\xi_i\right|\le \rho^{-4} \sum_{i=1}^m |n_i|\le \rho^{-5}\max_i|n_i|$, and since linear combinations of linear independent vectors are unique, we have that the set of linear combinations \[T :=\{\sum_{i=1}^m n_i \xi_i: n_i\in \Z, 0\le n_i \le\rho^{10-7m}\}\]is a set of at least $\rho^{10m-7m^2}$ vectors $\eta \in V\cap \Z^D$ with $|\eta|\le \rho^{5-7m}$ (recalling $\xi_1,\dots,\xi_m$ are linearly independent elements of $V\cap \Z^D$). We conclude the proof of our claim by taking $U = \pi( \{\rho^{7m}\eta: \eta \in T\})$.
\end{proof}
\end{clm}

By definition of $U$ (and the fact that $\phi$ is a homomorphism), for distinct $u,u' \in U$ and any $x \in \T^D$ we have that $\phi(x+u)$ and $\phi(x+u')$ belong to different boxes of sidelength $\rho^{7m}$. Furthermore, since a constant shift of the uniform distribution remains uniformly distributed, we have that $\phi(x_j+u)$ is uniformly distributed for each $u \in U$ and $j \in S_1$. Applying the principle of inclusion-exclusion, for any $j \in S_1$ we have
\begin{align*}
    p&:=\Prob( \phi(x_j+u) \in B\textrm{ for some } u\in U) &\\
    &= \sum_{u \in U} \Prob(\phi(x_j+u) \in B)- \sum_{U' \subset U, |U'| > 1} (-1)^{|U'|}\Prob(\phi(x_j+u') \in B\textrm{ for all } u'\in U') &\\
    &= \sum_{u \in U} \Prob(\phi(x_j+u) \in B)- \sum_{U' \subset U, |U'| > 1} 0\\
    &= |U| \operatorname{Vol}(B).\\
\end{align*}
\noindent Here the second-to-last line used the fact that $\phi(x+u),\phi(x+u')$ always belong to distinct boxes for distinct $u,u'\in U$, and the last line used the fact that $\phi(x_j+u)$ was uniformly distributed for each $u$. Recalling $B$ is an $m$-dimensional box with sidelengths $\rho^{7m}$, we know it has volume $\rho^{7m^2}$. Meanwhile, since $|U| \ge \rho^{10m-7m^2}$, we see $p \ge \rho^{10m}$. 

It follows that for each $j\in S_1$, \[\Prob(\phi(x_j+u)\not\in B \textrm{ for all }u \in U) \le 1-\rho^{10m}.\] By independence, \[\Prob(\phi(x_j+u)\not\in B \text{ for all } j \in S_1, u \in U\subset \pi(B_{\rho/10}(0))) \le (1-\rho^{10m})^{|S_1|}.\] Recalling $|S_1| = \rho^{-D/4}, 10m \le 10D/100\le D/8$, the RHS above is \[(1-\rho^{10m})^{\rho^{-D/4}}\le \exp(-\rho^{10m}\rho^{-D/4})\le \exp(-\rho^{-D/8}).\] By union bound, considering the $Y=\rho^{-cD}\le \rho^{-D^2}$ events $E_1,\dots,E_Y$, the $\le \rho^{-D^2}$ choices of $V$, and the $\le \rho^{-D^2}$ choices of $B$, we get that \[\Prob(F_2) \le \rho^{-3D^2} \exp(-\rho^{-D/8}). \]For large $D$, the above is less than $1/4$, completing the claim (to see the inequality, we let $Z$ denote $1/\rho>D^2>D$, and observe that the above equals $Z^{3D^2}\exp(-Z^{D/8})$ which can be seen to tend quickly towards zero).
\end{proof}
\end{prp} 

\subsection{Random annuli}\label{random annuli}

Rather than use ellipsoidal annuli with eccentricity chosen randomly as in \cite{green}, we use circular annuli with random radii. Also, instead of letting $N$ be any sufficiently large integer with respect to $D$, we require that $N$ is a specific integer which is appropriately large with respect to $D$ (see Remark~\ref{DwithN} for further discussion on this point).

For the rest of this paper, we some fix some choice of $x_1,\dots,x_M$ as guaranteed by Proposition~\ref{welld} (here we take $\rho = D^{-4}$, as specified in the beginning of Section~\ref{setup}).

Let $K = \lfloor \rho N^{4/D} \rfloor $. For $k \in \{0,\dots, K\}$, we let $A_k = \{x \in \R^D:kN^{-4/D}\le||x||_2<(k+1)N^{-4/D}\}$. We will pick $\mathbf{e}=(\mathbf{e}_1,\mathbf{e}_2,\dots,\mathbf{e}_M)$ from $\{0,\dots,K\}^M$ uniformly at random, and will also pick $\theta\in \T^D$ uniformly at random.

Given $\mathbf{e}\in \{0,\dots,K\}^M, \theta \in \T^D$, we let
\[\Blue_{\mathbf{e},\theta} =\{n \in [N]: n\theta \in \bigcup_{i =1}^M (x_i + \pi(A_{\mathbf{e}_i}))\}\]
\[\Red_{\mathbf{e},\theta} = [N]\setminus \Blue_{\mathbf{e},\theta}.\]
We wish to prove 
\begin{prp}\label{blue}We have
\[\Prob_\theta(\Blue_{\mathbf{e},\theta}\textrm{ contains a $3$-AP for some }\mathbf{e}\in \{0,\dots,K\}^M) \le 1/N.\]
\end{prp}

\begin{prp}\label{red}For any $\theta\in \Theta$,
\[\Prob_{\mathbf{e}}(\Red_{\mathbf{e},\theta}\textrm{ contains a $X$-AP}) \le 1/N.\]
\end{prp}
\noindent Here $\Theta \subset \T^D$ will be such that $\mu_{\T^D}(\Theta) \ge 1-  1/N$. We defer our exact definition of $\Theta$ to Section~\ref{diophantine}, but essentially $\theta \in \T^D$ will not belong to $\Theta$ if the orbit $\{\theta n: n \in [N]\}$ can contain long subprogressions that are concentrated around slices of high codimension.

It clearly follows that our strategy succeeds with probability at least $1-3/N$, proving Theorem~\ref{main}. 

Indeed, for sufficiently large $D$, if we choose $\theta \in \T^D$ randomly, then with probability at least $1-2/N > 0$ we will have that $\theta$ satisfies Proposition~\ref{blue} and also $\theta \in \Theta$. So we may fix such a $\theta$. Choosing $\mathbf{e} \in \{0,\dots,K\}^M$ randomly, with probability at least $1-1/N > 0$, $\mathbf{e}$ satisfies Proposition~\ref{red}, and so we may fix such a $\mathbf{e}$. So by Propositions~\ref{blue} and \ref{red} we have $\Blue_{\mathbf{e},\theta}$ lacks 3-APs and $\Red_{\mathbf{e},\theta}$ lacks $X$-APs. Using this blue-red coloring we get $f(N) \le X = N^{O(1/D)}$ (i.e., $f(D^{cD^2/2})) \le D^{O(D)} $) and since $f$ is increasing we are done.

\begin{rmk}\label{DwithN}
An interesting artefact of the approach we present will be that if we fix $D$ and let $N$ get too large, then our $D$-dimensional strategy does not (necessarily) work; it will be essential\footnote{Here, we mean that our proof fails unless $D$ grows with $N$.} to have $D$ grow with $N$ to prove non-trivial bounds on $f(N)$ (and consequently $w(3,k)$). 

This is in contrast to how the $D$-dimensional strategy in \cite{green} works for all sufficiently large $N$. Thus, one may view the methods of this paper as ``unstable'' and the methods of \cite{green} as ``stable''.

Originally, the author believed that the techniques of this paper did not have a ``stable analogue'', in some well-defined sense. This was because our proof requires $Y=\rho^{-cD}$ to be greater than $K \approx N^{4/D}$. 

However, we later realized that if we let $\rho$ decrease with $N$ (e.g., it would work for $\rho$ to be roughly $N^{-5c^{-1}/D^2}$), we can continue to avoid red progressions of length $N^{O(1/D)}$. For simplicity, we decided to write-up the unstable approach where $\rho$ is fixed given $D$. Obtaining the stable analogue is completely straightforward from our work, the only change that must be made is in Section~\ref{diophantine} where $D^{C_1}$ should be replaced by $\rho^{-C_1}$ in the definition of $\Theta$.

\end{rmk}

\section{Blue APs}\label{bluecase}
Our proof of Proposition~\ref{blue} works the same as the proof of Proposition~5.1 from \cite[p. 17-18]{green}. The author decided to reorganize the argument slightly for aesthetic reasons.

We need the following quantitative result concerning progressions in an annulus. For completeness, we include its proof, which has also appeared in \cite{green,wolf}.
\begin{lem}\label{smallv}Suppose $u,v \in \R^d$ are such that $\{u,u+v,u+2v\}\subset A_k$ for some $k\in \{0,\dots,K\}$. Then $||v||_2 < \frac{1}{2}N^{-2/D}$.
\begin{proof}
Recall the parallelogram law, which states
\[2||v||_2^2 = ||u||_2^2+||u+2v||_2^2-2||u+v||_2^2.\]It follows that \begin{align*}
    ||v||_2^2 &\le \left((k+1)N^{-4/D}\right)^2-(kN^{-4/D})^2\\
    &=(2k+1)N^{-8/D}\\
    &\le (2K+1)N^{-8/D}\\
    &\le (2\rho+N^{-4/D}) N^{-4/D} .\\
\end{align*}For sufficiently large $D$, $2\rho+N^{-4/D}<1/4$. Hence the result follows.
\end{proof}
\end{lem}

We shall use this to prove the following.
\begin{prp}\label{ifblue}Suppose $\theta \in \T^D$ is such that $||d\theta||_{\T^D}>\frac{1}{2}N^{-2/D}$ for all $d\in [N]$. Then, for all $\mathbf{e} \in \{0,\dots,K\}^M$, $\Blue_{\mathbf{e},\theta}$ contains no blue $3$-APs.
\end{prp}With this, we may quickly prove Proposition~\ref{blue}.

\begin{proof}[Proof of Proposition~\ref{blue} assuming Proposition~\ref{ifblue}]
Let $\theta \in \T^D$ be chosen uniformly at random. 

Since $d\theta $ is uniformly distributed for each $0<d\le N$ and $\{x:||x||_{\T^D}\le \frac{1}{2}N^{-2/D}\}$ has volume $N^{-2}$, by union bound \[\Prob_\theta(||d\theta||_{\T^D} > \frac{1}{2}N^{-2/D}\textrm{ for all }d \in [N]) \ge 1-1/N.\]Conditioning on this event, we have that $\Blue_{\mathbf{e},\theta}$ will not have any $3$-APs for every choice of $\mathbf{e} \in \{0,\dots,K\}^M$. Since this happens with probability $1-1/N$, we are done.
\end{proof}

\begin{proof}[Proof of Proposition~\ref{ifblue}]Let $\theta,\mathbf{e}$ be fixed. We seek to prove that if there is a blue $3$-AP, then $||d\theta||_{\T^D}\le \frac{1}{2}N^{-2/D}$ for some $d \in [N]$.

For any blue $3$-AP, $P = \{n,n+d,n+2d\}\subset \Blue_{\mathbf{e},\theta}$, we first claim that $\theta P \subset x_i +\pi(A_{\mathbf{e}_i})$ for some $i\in [M]$. 

Indeed, as $P$ is blue, there must be $i,j,k \in [M]$ such that $\theta n \in x_i+ \pi(A_{\mathbf{e}_i}), \theta(n+d) \in x_j + \pi(A_{\mathbf{e}_j}), \theta (n+2d) \in x_k+\pi(A_{\mathbf{e}_k})$. Now, since $\theta(n+2d) -2\theta (n+d) +\theta n = 0$, we should have that \begin{equation}
    x_i-2x_j+x_k \in \pi(A_{\mathbf{e}_i})-2\pi(A_{\mathbf{e}_j})+\pi(A_{\mathbf{e}_k}) = \pi(A_{\mathbf{e}_i}-2A_{\mathbf{e}_j}+A_{\mathbf{e}_k})\label{blue container}
\end{equation}(to get the RHS, we use the fact that sumsets are preserved under projections/homomorphisms). Recalling that $A_{\mathbf{e}_t} \subset B_\rho(0)$ for every $t\in [M]$ and applying triangle inequality, the RHS of Equation~\ref{blue container} is a subset of $\pi(B_{4\rho}(0))$. Hence, we get that $||x_i-2x_j+x_k||_{\T^D} \le 4\rho$ (since $x \in \pi(B_{4\rho}(0))$ implies that $4\rho \ge ||\pi^{-1}(x)||_2 \ge ||\pi^{-1}(x)||_\infty =||x||_{\T^D}$). By the design of $x_1,\dots,x_M$ (in particular, due to Proposition~\ref{welld} Condition 1), this can only occur if $i = j = k$.

So, let $i$ be such that $\theta P \subset x_i + \pi(A_{\mathbf{e}_i})$. Taking $u = \pi^{-1}(\theta n -x_i),v = \pi^{-1}(\theta d)$, we wish to apply Lemma~\ref{smallv}. Indeed, if we can confirm the hypothesis (namely $\{u,u+v,u+2v\} \subset A_{\mathbf{e}_i}$), then the conclusion of Lemma~\ref{smallv} gives $\frac{1}{2}N^{-2/D} \ge ||v||_2 \ge ||v||_\infty = ||d \theta ||_{\T^D}$. In which case we are done, since $d \in [N]$ (because $P \subset [N]$) and we have just observed $||d\theta||_{\T^D}$ is small.

We conclude by confirming the hypothesis that $\{u,u+v,u+2v\}\subset A_{\mathbf{e}_i}$. Given $\rho < 1/12$ (which is true whenever $D\ge 2$), we will have that $A_{\mathbf{e}_i}\subset B_\rho(0)\subset B_{1/12}(0)$. Now set $B:= B_{1/12}(0)$. It is clear that $\pi(u) \in \pi(B), \pi(v) \in \pi(B-B)$ and consequently $\pi(\{u,u+v,u+2v\})\subset \pi(5B)$. With $S := (-1/2,1/2]^D = \pi^{-1}(\T^D)$, since $\pi|_S$ is a bijection and $S \supset 5B$, it follows that $\pi^{-1}(\pi(\{u,u+v,u+2v\})) = \{u,u+v,u+2v\}$ (thus there is no wrap-around issue).
\end{proof}
\section{Diophantine conditions}\label{diophantine}
Let $C_1,C_2$ be some fixed positive constants where $C_2 \ge 4(c/2)^{-1}C_1 =800C_1$. And let $X= N^{100(C_2+2)/D}$ so that $X^{D/100} \ge D^{4C_1D^2/100} N^2$. 

For our purposes, we define $\Theta\subset \T^D$ to be the set of all $\theta\in \T^D$ such that \[\dim(\{\xi\in \Z^D:|\xi|<D^{C_1},||n\xi\cdot\theta||_\T <D^{C_1D}X^{-1}\})< D/100\] for each $n \in [N]$. Following the vocabulary of \cite[Section~7]{green}, $\Theta$ is our set of ``diophantine'' rotations. We will show that $\mu_{\T^D}(\Theta) = 1-o(1)$ as $D\to \infty$, meaning that generic $\theta \in \T^D$ are diophantine.

\begin{rmk}
We note that we work with a weaker notion of $\theta \in \T^D$ being ``diophantine'' than what was used in \cite[Section~7]{green}. Indeed, in \cite{green}, $\theta \in \T^D$ was called diophantine if $\theta \in \Theta$ and also for all $d \in [N/X]$,
\[\#\{n\in[X] : ||nd \theta ||_{\T^D} \le X^{-1/D}\}\le X^{9/10}.\]

This additional condition is no longer necessary for our purposes. This is another convenient simplification, as it was a bit cumbersome to prove that generic $\theta \in \T^D$ satisfied this second condition.
\end{rmk}

Here is some brief intuition for our definition of $\Theta$. Essentially it means that any sufficiently long progression with common difference $\alpha = \theta d$ for some $d\in [N]$, that the progression cannot be too concentrated around a subtorii with codimension $\ge D/100$ (as mentioned in the end of Subsection~\ref{how}, allowing such concentration would cause complications).

\begin{prp}\label{diop}For sufficiently large $D$, we have $\mu_{\T^D}(\Theta) > 1-1/N$.
\begin{proof}The work here is essentially the same as in the first part of the proof of Proposition~7.1 from \cite[p. 19]{green}. We remind the reader that $X^{D/100} \ge D^{4C_1D^2/100}N^2$ (by the assumptions at the start of this section).

Fix any linearly independent $\xi_1,\dots,\xi_{D/100} \in \Z^D$. For any $n\in [N]$ we have that \begin{align*}
\mu_{\T^D}(\{\theta:||n\xi_i\cdot\theta||_\T \le D^{C_1D}X^{-1}\textrm{ for all } i\in [D/100]\}) &= (2D^{C_1D})^{D/100}X^{-D/100}\\
&< D^{2C_1D^2/100}X^{-D/100}\\
&\le D^{-2C_1D^2/100} N^{-2}.\\
\end{align*}Thus, by union bound the measure of $\theta$ where this holds for some $n\in [N]$ is at most $D^{-2C_1D^2/100}N^{-1}$.

Now, if $\theta \not\in \Theta$, then there exists linearly independent $\xi_1,\dots,\xi_{D/100}\in \Z^D,|\xi_i|\le D^{C_1}$ such that the preceding statement holds. The number of choices of $\xi_1,\dots,\xi_{D/100}$ is at most $(3D^{C_1})^{D^2/100}<D^{2C_1D^2/100}$. Hence, by union bound, $\mu_{\T^D}(\T^D\setminus \Theta) < N^{-1}$ giving the result.
\end{proof}
\end{prp}

\section{Red APs}\label{final}

Let $X=N^{100(C_2+2)/D}$ and $Y = \rho^{-cD}$. In both results we are assuming $D$ is sufficiently large.

\begin{prp}\label{step1}For any $\theta \in \Theta$, and any $X$-AP, $P = \{n_0,n_0+d,\dots,n_0+d(X-1)\}\subset [N]$, there exists distinct $j_1,\dots j_Y \in [M]$ such that for each $k \in [Y]$, $\theta P \cap (x_{j_k}+\pi(B_{\rho/5}(0))) \neq \emptyset$.
\begin{proof}We mimic the proof of Proposition~5.3 from \cite[p. 22-23]{green}, but make appropriate changes to make use of our strengthened version of Proposition~\ref{welld}.

Set $\alpha = \theta d$. Let $n_1 = n_0 +d\lfloor X/2\rfloor$. Note that $P \supset \{n_1+id:i\in [-X/5,X/5]\cap \Z\}$. Set
\[\Lambda = \{\xi \in \Z^D:|\xi|<\rho^{-3},||\xi\cdot \alpha||_\T \le \rho^{-2D}X^{-1}\}.\]

By the definition of $\Theta$, and taking $C_1 \ge 12$, it follows that $\dim(\Lambda)<D/100$. By Proposition~\ref{welld} Condition 2, there exists distinct $j_1,\dots,j_Y \in [M]$ and $u_1,\dots,u_Y \in \pi(B_{\rho/10}(0))$ such that
\begin{equation}
    ||\xi \cdot(x_{j_k}+u_k-\theta n_1)||_\T\le 1/100 \label{close}
\end{equation}
for all $\xi \in \Lambda,k\in [Y]$. We claim that for each $k\in [Y]$, $\theta P$ intersects $x_{j_k}+u_k+\pi(B_{\rho/10}(0))\subset x_{j_k}+\pi(B_{\rho/5}(0))$, which will conclude our proof.

We will fix $k \in [Y]$ and deduce $\theta P$ intersects $x_{j_k}+u_k+\pi(B_{\rho/10}(0))$. We are then done by repeating our argument for each $k \in [Y]$.

We proceed by a Fourier analytic argument which will make use of Equation~\ref{close}. We will also need two functions $\chi,w$ which we specify below. First, we need $\chi:\T^D \to \R$ so that:
\begin{enumerate}[label={(\arabic*a)}]
    \item \label{1a} $\chi(x)\le 0$ unless $x \in \pi(B_{\rho/10}(0))$
    \item \label{2a} $\hat{\chi}$ only takes non-negative real values
    \item \label{3a} $\hat{\chi}$ is only supported on $\{\xi:|\xi| \le \rho^{-3}\}$
    \item \label{4a} $\int \chi = 1$
    \item \label{5a} $\int |\chi| \le 3$.
\end{enumerate}
We provide such a function in Lemma~\ref{con chi}, the construction is from Lemma~B.6 of \cite{green}. Next, we need $w: \Z\to [0,\infty)$ so that:
\begin{enumerate}[label={(\arabic*b)}]
    \item \label{1b} $w$ is only supported on $[-X/5,X/5]$
    \item \label{2b} $\hat{w}$ only takes non-negative real values
    \item \label{3b} $\sum_{n \in \Z} w(n) \ge X$
    \item \label{4b} $|\hat{w}(\beta)|\le 2^{7}X^{-1} ||\beta||_{\T}^{-2}$ for all $\beta \in \T$.
\end{enumerate}
For this we can use a Fej\'er kernel (see Lemma~\ref{con w} for details).

It will suffice to prove
\begin{equation}
    \sum_{i \in \Z} w(i)\chi(\theta(n_1+id)-u_k-x_{j_k})>0\label{fourier goal}.
\end{equation}Indeed, assuming Equation~\ref{fourier goal} holds true, it will follow that there exists $i_k\in \Z$ where the summand is positive. As $w$ does not take negative values, it follows that $w(i_k)> 0$ and $\chi(\theta(n_1+i_kd)-u_k-x_{j_k})>0$ will hold. By Property~\ref{1b} we will have $i_k \in [-X/5,X/5]\cap\Z$, and by Property~\ref{1a} we will have $\theta (n_1+i_kd) \in x_{j_k}+u_k +\pi(B_{\rho/10}(0))$. It is clear then that $n'_k:=n_1+i_kd \in P$, and $\theta n'_k \in x_{j_k} + u_k +\pi(B_{\rho/10}(0))$, as desired.

We shall now prove Equation~\ref{fourier goal}. By Fourier inversion on $\chi$ (and recalling $\theta d = \alpha$), the LHS of Equation~\ref{fourier goal} may be written as

\begin{align*}
    \hspace{5em}&\hspace{-5em}\sum_{\xi \in \Z^D}\hat{\chi}(\xi)e(\xi \cdot (\theta n_1-u_k-x_{j_k}))\sum_{i\in\Z} w(i)e(i\xi \cdot \alpha)\\
    &= \sum_{\xi \in \Z^D}\hat{\chi}(\xi)e(\xi \cdot (\theta n_1-u_k-x_{j_k}))\hat{w}(-\xi\cdot \alpha)\\
    &= \sum_{\xi \in \Z^D}\hat{\chi}(\xi)\cos(2\pi \xi \cdot (\theta n_1-u_k-x_{j_k}))\hat{w}(-\xi\cdot \alpha)\numberthis\label{inversion sum}.\\
\end{align*}In the last line we took real parts, recalling that the LHS was real as were $\hat{\chi}$ and $\hat{w}$ (due to Properties~\ref{2a} and \ref{2b}).

By Property~\ref{3a}, we have that the summands of \ref{inversion sum} are only supported where $|\xi|\le \rho^{-3}$. We now consider the contributions of these summands which we split into three cases of $\xi$: (i) $\xi = 0$, (ii) $\xi \in \Lambda$, and (iii) $|\xi|\le \rho^{-3}$ yet $\xi \not \in \Lambda$.

(i): The contribution from $\xi =0$ is precisely $\hat{\chi}(0)\hat{w}(0) = (\int \chi)(\sum_{n\in \Z}w(n))$, which is $\ge X$ by Properties~\ref{4a} and \ref{3b}.

(ii): Here, $\xi \in \Lambda$. Thus, by our choice of $j_k,u_k$ (c.f. Equation~\ref{close}), we will have $\cos(2\pi \xi \cdot (\theta n_1 -u_k-x_{j_k})) >0$. By Properties~\ref{2a} and \ref{2b}, $\hat{\chi}$ and $\hat{w}$ only take non-negative real values, thus the contribution of all these terms will be non-negative.

(iii): By triangle inequality, the absolute value of the contribution to \ref{inversion sum} from these $\xi$ is at most
\begin{equation}
    \sum_{|\xi| \le \rho^{-3},\xi \not \in \Lambda} |\hat{\chi}(\xi)||\hat{w}(-\xi\cdot \alpha)| \label{third sum magnitude}.
\end{equation} Now, for any $\xi\in \Z^D$, we have $|\hat{\chi}(\xi)| \le \int |\chi| \le 3$ (applying Property~\ref{5a} for the last inequality). Meanwhile for any $\xi \in \Z^D \setminus \Lambda$ with $|\xi| \le \rho^{-3}$, we have 
\[|\hat{w}(-\xi\cdot \alpha )| \le 2^7X^{-1}||- \xi\cdot \alpha||_\T^{-2}\le 2^7 \rho^{4D}X, \] by applying Property~\ref{4b} for the first inequality and the definition of $\Lambda$ in the second inequality.

Hence, no summand from Equation~\ref{third sum magnitude} is larger than $2^9\rho^{4D}X$. Meanwhile, the number of $\xi \in \Z^D$ satisfying $|\xi|\le \rho^{-3}$ is at most $(3\rho^{-3})^D$ which assuming $D \ge 10$ is at most $2^{-10}\rho^{-4D}$. It follows that Equation~\ref{third sum magnitude} is at most $X/2$, meaning the contribution from (iii) is at most $X/2$ in magnitude.

Putting these three contributions together, we get that Equation~\ref{inversion sum} is at least $X+0-X/2 = X/2 > 0$. Hence, we see Equation~\ref{fourier goal} is indeed positive, completing our proof.
\end{proof}
\end{prp}
We now can prove Proposition~\ref{red}. Let $K = \lfloor \rho N^{4/D}\rfloor\le N^{4/D}-1$.
\begin{proof}We intend to do a union bound over all $X$-APs $P\subset [N]$. Fix some $X$-AP $P \subset[N]$ and some $\theta \in \Theta$.

By Proposition~\ref{step1}, there exists distinct $j_1,\dots,j_Y\in [M]$, $n'_1,\dots,n'_Y \in P$ such that $\theta n'_i \in x_{j_i} +\pi(B_{\rho/5}(0))$ for each $i \in [Y]$. For each $i \in [Y]$, let $d_i = ||\pi^{-1}(x_{j_i} - \theta n'_i)||_2$, and observe $d_i \in [0,\rho/5)$. 
For each $i \in [Y]$, there exists $k \in \{0,\dots,K\}$ such that $d_i \in [kN^{-4/D},(k+1)N^{-4/D})$, let $k_i$ denote this $k$. Now, $\mathbf{e}$ samples $\{0,\dots,K\}^M$ uniformly at random, hence
\begin{align*}\Prob_\mathbf{e}(\mathbf{e}_{j_i} \neq k_i \textrm{ for all }i \in [Y])&=(1-1/(K+1))^Y\\
&\le (1-N^{-4/D})^{\rho^{-cD}}\\
&= (1-D^{-2cD})^{D^{4cD}}\\
&\le \exp(-D^{2cD}) \\
&\le N^{-3}.\\\end{align*}where in the last line we use the fact that $D^{2cD} > \log(N^3) = O(\log(D)D^2)$.

Now, for $i\in [Y]$, if $\mathbf{e}_{j_i} = k_i$, then we have that $\theta n'_i \in x_{j_i} + \pi(A_{\mathbf{e}_{j_i}})$ meaning $P$ is not a red AP, as desired. By a union bound, as there are less than $N^2$ different $X$-APs $P\subset [N]$ to consider, we see that $\Prob_\mathbf{e}(\Red_{\mathbf{e},\theta} \textrm{ has an }X\textrm{-AP})\le N^{-1}$ as desired.

\end{proof}

\begin{appendix}
\section{Lattices}\label{ax lattice}

In Appendix~\ref{ax lattice}, we continue to use the convention from the main paper that for $\xi \in \Z^D$, we write $|\xi|$ to denote $||\xi||_\infty$.

We must recall an old result from lattice reduction theory. The following was originally proven by Mahler \cite{mahler}, but some may find \cite[Lemma~8 in Section~V.4]{cassels} a more accessible reference.
\begin{lem}\label{short basis}\cite{cassels,mahler} Let $\Lambda$ be a lattice with rank $m$, with $x_1,\dots,x_m \in \Lambda$ being linearly independent. Then there exists an integral basis $w_1,\dots,w_m$ of $\Lambda$ such that $\max_{i\in [m]} |w_i| \le \frac{m}{2} \max_{i\in [m]}|x_i|$.
\end{lem}
We can now prove a necessary lemma.
\begin{lem}\label{lattice gen} Let $Q\ge 1 $ be a parameter. Let $V \le \Q^D$ be a subspace spanned over $\Q$ by linearly independent vectors $v_1,\dots,v_m\in \Z^D$ where $|v_i|\le Q$ for each $i\in [m]$. Then there exists vectors $w^{(1)},\dots,w^{(m)} \in V\cap \Z^D\cap [-QD,QD]^D$ such that every $x\in V\cap \Z^D$ with $|x|\le Q$ is a (unique) $\Z$-linear combination $x = \sum_{i=1}^m n_i w^{(i)}$ with $|n_i| \le m!(DQ)^m$.
\begin{proof}We present a slight modification of the proof of Lemma~A.2 from \cite[p. 62]{green}.

Write $X := \{x \in V \cap \Z^D: |x| \le Q\}$. Let $\Lambda$ denote the lattice generated by $X$ (meaning $v\in \Lambda$ if and only if it can be written as a $\Z$-linear combination $v= \sum_{x\in X}n_x x$). Clearly, since we assumed $v^{(1)},\dots,v^{(m)} \in X$ are linearly independent and span $V\supset \Lambda$ (over $\Q$), we see that $\Lambda$ has rank $m$.

Applying Lemma~\ref{short basis} to $\Lambda$, we find an integral basis $w^{(1)},\dots,w^{(m)}$ of $\Lambda$ such that $\max_{i \in [m]}|w^{(i)}| \le D\max_{i\in [m]}|v_i| \le DQ$.

Consider the $m\times D$-matrix whose $(i,j)$-entry is the $j$-coordinate $w_j^{(i)}$. Since $w^{(1)},\dots,w^{(m)}$ are linearly independent and are the column vectors of our matrix, said matrix has rank $m$ and thus has an invertible $m\times m$-minor. By relabelling axes, we may without loss of generality assume that $(w_j^{(i)})_{1\le i,j \le m}$ is such a minor. We will set $A$ to be the inverse of $(w_j^{(i)})_{1\le i,j \le m}$.

Now, for $x \in \Lambda$, we have $x= \sum_{i=1}^m n_iw^{(i)}$ for some unique choice of integers $n_1,\dots,n_m$. In particular, we will have that $A(x_1,\dots,x_m)^T = (n_1,\dots,n_m)^T$. By the formula for the inverse of a matrix in terms of its adjugate, we get that entries of $A$ belong to the set $\{\frac{a}{q}: a \in \Z, |a|\le (m-1)!(DQ)^{m-1}\}$, where $q := \det((w_j^{(i)})_{1\le i,j \le m})\in \Z\setminus \{0\}$. Thus, we will have that each $n_i$ will be an integer with absolute value at most $m!(DQ)^{m-1}|x|$.

So, when $x\in X \subset \Lambda$, we get that $|n_i| \le m!(DQ)^{m-1}Q\le m!(DQ)^m$ as desired.
 
\end{proof}
\end{lem}

\section{Fourier Analysis}\label{ax fourier}
In this appendix, we construct two cutoff functions, for use in Proposition~\ref{step1}. These cutoff functions are both variants of the classical Fej\'er kernel construction. 

In the main paper, we wrote $\pi$ to denote the natural projection from $\R^D\to \T^D$. In Appendix~\ref{ax fourier}, we will continue to use this projection, but will also sometimes write $\pi$ to denote the constant $3.141592\dots$. Our usage should always be clear from context.

\begin{lem}\label{con w}Let $X\ge 1$. There exists $w:\Z\to [0,\infty)$ so that:
\begin{enumerate}
    \item \label{1B} $w$ is only supported on $[-X/5,X/5]$
    \item \label{2B} $\hat{w}$ only takes non-negative real values
    \item \label{3B} $\sum_{n \in \Z} w(n) \ge X$
    \item \label{4B} $|\hat{w}(\beta)|\le 2^{7}X^{-1} ||\beta||_{\T}^{-2}$ for all $\beta \in \T$.
\end{enumerate}
\begin{proof}This is a standard Fej\'er kernel construction. For $S\subset \R$, let $\mathds{1}_S$ to denote the indicator function defined from $\Z\to \{0,1\}$ such that $\mathds{1}_S(n) =1\iff n \in S$. 

We take
\[w:= \frac{100}{X}\mathds{1}_{[-X/10,X/10]}*\mathds{1}_{[-X/10,X/10]}(n).\]It is clear that $w$ only takes non-negative values, thus $w$ is a function from $\Z\to [0,\infty)$, it remains to verify its properties.

It is straightforward to see that Properties~\ref{1B}, \ref{2B}, and \ref{3B} all hold. Indeed, it suffices to confirm that $\supp(w) = [-2\lfloor X/10 \rfloor,2\lfloor X/10\rfloor] \subset [-X/5,X/5]$, $\sum_{n\in \Z}w(n) = \frac{100}{X} |\Z \cap [-\lfloor X/10 \rfloor,\lfloor X/10\rfloor]|^2 \ge X$, and $w(n) = w(-n)$ respectively are true. 

For Property~\ref{4B}, we evaluate the Fourier transform explicitly as
\[\hat{w}(\beta)= \frac{100}{X}\left|\sum_{|n|\le X/10} e(-\beta n)\right|^2.\]When $\beta = 0$, we consider Property~\ref{4B} to be vacuously true, so we may assume $\beta \in \T \setminus \{0\}$. Evaluating the above geometric sum for $\beta \in \T\setminus\{0\}$, we get\[\left| \sum_{|n|\le  X/10}e(-\beta n) \right|\le \frac{2}{|1- e(\beta)|} = \frac{1}{|\sin \pi\beta|} \le ||\beta||_\T^{-1}.\]
The result follows.

\end{proof}
\end{lem}
We remind our readers that $B_\ep(0)$ denotes the Euclidean ball of radius $\epsilon$ in $\R^D$, and that $\pi:\R^D\to \T^D$ denotes the natural projection. 
\begin{lem}\label{con chi}
Let $D$ be sufficiently large and $\rho \le D^{-4}$. There exists $\chi:\T^D \to \R$ so that:
\begin{enumerate}
    \item \label{1A} $\chi(x)\le 0$ unless $x \in \pi(B_{\rho/10}(0))$
    \item \label{2A} $\hat{\chi}$ only takes non-negative real values
    \item \label{3A} $\hat{\chi}$ is only supported on $\{\xi:|\xi| \le \rho^{-3}\}$
    \item \label{4A} $\int \chi = 1$
    \item \label{5A} $\int |\chi| \le 3$.
\end{enumerate}
\begin{proof} We repeat the proof of Proposition~B.6 from \cite[p. 66-67]{green}.

Rather than work with Euclidean balls and the $\ell^2$-norm, it will be more convenient to work with the distance $||\cdot ||_\T$ directly. Since $D^{1/2} ||\pi(t)||_\T \ge ||t||_2$ for $t \in (-1/4,1/4]^D\supset B_{\rho/10}(0)$, thus we can replace Property~\ref{1A} with the stronger property that $\chi\le 0$ unless $||x||_{\T^D}\le \rho^{7/6} \le \rho D^{-1/2}/10$.

Let $k = \lfloor \rho^{-3}\rfloor$. Consider
\begin{align*}
    \psi(x)&:= (2D+\sum_{i=1}^D (e(x_i)-e(-x_i)))^k-4^k(D-\rho^{7/3})^k\\
    &= 4^k(\cos^2(\pi x_1) + \dots +\cos^2(\pi x_D))^k-4^k(D-\rho^{7/3})^k.\\
\end{align*}Since $\cos(\pi t) \le 1-t^2$ for $|t|\le1/2$, we have that if $||x||_{\T^D} \ge \rho^{7/6}$ then\[0\le \cos^2(\pi x_1) + \dots+\cos^2(\pi x_D) \le D-\rho^{7/3}\] causing $\psi(x)\le 0$. By expanding out the first definition of $\psi$, it is clear to see that $\hat{\psi}(\xi)$ is only supported on $|\xi|\le k$ (in fact on $||\xi||_1\le k$), and $\hat{\psi}(\xi) \ge 0$ for all $\xi \in \Z^D \setminus \{0\}$. 

We next wish to show that $\int \psi = \hat{\psi}(0) > 0$. Using the inequality $\cos^2(\pi t) \ge 1-\pi^2t^2$ for $|t|\le 1/2$, we conclude that if $||x_i||_{\T}\le \frac{1}{4 \sqrt{k}}$ for all $i\in [D]$ then \[ (\cos^2(\pi x_1) +\dots + \cos^2(\pi x_D))^k \ge D^k(1-1/k)^k \ge \frac{1}{3}D^k .\]Therefore\[\int_{\T^D} (\cos^2(\pi x_1) +\dots + \cos^2(\pi x_D))^k \ge \left(\frac{1}{2\sqrt{k}}\right)^D \frac{1}{3}D^k > 2 k^{-D}D^k.\]

Meanwhile, using the facts that $k = \lfloor \rho^{-3}\rfloor,\rho \le D^{-4}$, and $D$ is sufficiently large, we have that
\begin{equation}(D-\rho^{7/3})^k \le D^k \exp(- \rho^{7/3}k/D) < D^k k^{-D}.\label{bounded}\end{equation} Combining this with our previous paragraph, we have that $\int \psi > k^{-D} (4D)^k$.

We therefore take $\chi := \frac{\psi}{\int \psi}$. It is immediately clear that $\chi$ satisfies Properties~\ref{1A}, \ref{2A}, \ref{3A}, and \ref{4A}, leaving us to confirm Property~\ref{5A} also holds.

We write $\psi = \psi_+ - \psi_-$ in positive and negative parts. For all $x\in \T^D$, we have that $\psi_-(x) \le 4^k(D-\rho^{7/3})^k < k^{-D} (4D)^k$ (where in the last inequality we reuse Equation~\ref{bounded}). Hence, $\int \psi_-\le k^{-D}(4D)^k <\int \psi$. Since $|\psi| = \psi + 2\psi_-$, we have that $\int |\psi| \le 3 \int \psi$. Thus, we conclude Property~\ref{5A} also holds and we are done.
\end{proof}
\end{lem}

\end{appendix}

\end{document}